\documentclass[11pt]{article}
\usepackage[margin=1in]{geometry}
\usepackage[utf8]{inputenc}
\usepackage{amsthm}
\usepackage{nicefrac}
\usepackage{amssymb}
\usepackage{amsmath}
\usepackage{tikz-cd}
\usepackage{tikz}
\usepackage{bbm}
\usepackage{hyperref}

\theoremstyle{plain}
\newtheorem{theorem}{Theorem}[section]
\newtheorem{lemma}[theorem]{Lemma}
\newtheorem{proposition}[theorem]{Proposition}
\newtheorem{corollary}[theorem]{Corollary}

\newtheorem{condition}[theorem]{Condition}
\newtheorem*{thm*}{Theorem}
\newtheorem*{lemma*}{Lemma}
\newtheorem*{prop*}{Proposition}
\newtheorem*{cor*}{Corollary}
\newtheorem*{conj*}{Conjecture}

\theoremstyle{definition}
\newtheorem{definition}[theorem]{Definition}
\newtheorem{example}[theorem]{Example}

\theoremstyle{remark}
\newtheorem*{rmk}{Remark}

\newcommand{\comment}[1]{}

\renewcommand{\dim}[1]{\textsf{dim}\,#1}
\renewcommand{\ker}[1]{\textsf{Ker}\,#1}
\newcommand{\supp}[1]{\textsf{supp}\,#1}
\newcommand{\spn}[1]{\textsf{span}\,#1}
\newcommand{\rank}[1]{\textsf{rank}\,#1}
\newcommand{\diag}[1]{\textsf{diag}\,#1}
\newcommand{\im}[1]{\textsf{Im}\,#1}
\newcommand{\1}[1]{\mathbbm{1}\,#1}

\newcommand{\Acal}{\mathcal{A}}

\newcommand{\Newt}{\textsf{Newt}}
\renewcommand{\int}{\textsf{int}}
\newcommand{\relint}{\textsf{relint}}
\renewcommand{\supp}{\textsf{supp}}
\newcommand{\conv}{\textsf{conv}}
\newcommand{\mvol}{\textsf{MVol}}
\newcommand{\vol}{\textsf{Vol}}

\newcommand{\bfa}{\mathbf{a}}
\newcommand{\bfb}{\mathbf{b}}
\newcommand{\bfc}{\mathbf{c}}
\newcommand{\bfd}{\mathbf{d}}
\newcommand{\bfe}{\mathbf{e}}

\newcommand{\bfp}{\mathbf{p}}
\newcommand{\bfq}{\mathbf{q}}
\newcommand{\bfr}{\mathbf{r}}

\newcommand{\bfv}{\mathbf{v}}
\newcommand{\bfw}{\mathbf{w}}
\newcommand{\bfx}{\mathbf{x}}
\newcommand{\bfy}{\mathbf{y}}
\newcommand{\bfz}{\mathbf{z}}

\newcommand{\R}{\mathbb{R}}
\newcommand{\C}{\mathbb{C}}
\newcommand{\Z}{\mathbb{Z}}

\title{Mixed volumes of networks with binomial steady-states}

\author{Jane Ivy Coons\footnote{St John's College, University of Oxford and Mathematical Institute, University of Oxford, jane.coons@maths.ox.ac.uk}, Mark Curiel\footnote{Department of Mathematics, University of Hawai`i at M\={a}noa, curielm@hawaii.edu},  Elizabeth Gross\footnote{Department of Mathematics, University of Hawai`i at M\={a}noa, egross@hawaii.edu} }
\date{}

\begin{document}

\maketitle

\abstract{The steady-state degree of a chemical reaction network is the number of complex steady-states for generic rate constants and initial conditions.  One way to bound the steady-state degree is through the mixed volume of the steady-state system or an equivalent system. In this work, we show that for partionable binomial networks, whose resulting steady-state systems are given by a set of binomials and a set of linear (not necessarily binomial) conservation equations,  computing the mixed volume is equivalent to finding the volume of a single mixed cell that is the translate of a parallelotope.  We then turn our attention to identifying networks with binomial steady-state ideals. To this end, we give a coloring condition on directed cycles that guarantees the network has a binomial steady-state ideal, and consequently, toric steady-states. We highlight both of these theorems using a class of networks referred to as species-overlapping networks and give a formula for the mixed volume of these networks.  }

\section{Introduction}

Chemical reaction networks (CRNs) are graphs on a set of complexes (e.g. molecules) that visually summarize the interactions present in a chemical system. They are useful for modeling cellular biological processes such as signal transduction, and under a more general setting, are found in epidemiology and ecology.   Under the assumption of mass-action
kinetics, chemical reaction networks encode a system of \emph{polynomial} ordinary differential equations.  Understanding the number of possible (real, positive) stable steady-states of such a system is key to determining whether a given reaction network is an appropriate model for a given biological process.  Indeed, one highly active area of research in regards to chemical reaction networks is to develop criteria that guarantee or preclude  \emph{multistationarity}, the capacity for multiple real, positive  stable steady-states (see, e.g. \cite{JS2015}).

While multistationarity is determined by the number of possible positive stable steady-states over the reals, we can relax the definition of steady-state to include any complex solution to the steady-state equations obtained by setting each ODE equal to zero.  The number of complex steady-states for generic rate constants and
initial conditions is called the \emph{steady-state degree} of a chemical reaction network \cite{grosshill2021steady}. The steady-state degree is a bound on the number of real, positive steady-states.

In general, the steady-state degree can be challenging to determine, and so far, most results focus on a particular model or families of models.  For example, in \cite{GHRS16}, the authors show that the steady-state degree of the Wnt shuttle model is 9. One way to bound the steady-state degree of a chemical reaction network is through the mixed volume of their corresponding steady-state system.  For example, the mixed volume was used to bound the steady-state degree of a model of ERK regulation in \cite{OSTT2019} as well as for three family of networks in \cite{grosshill2021steady}, including multisite distributive phosphorylation networks.  

In this work, we focus on networks whose steady-state ideals are binomial, that is, generated by polynomials with at most two terms.  In the literature, networks with binomial steady-state ideals that admit a real solution are referred to networks with \emph{toric steady-states} \cite{PDSC2012}.   This work has two key theorems (Theorem \ref{thm:main} and Theorem \ref{thrm:toriccycles}). In  our first key theorem (Theorem \ref{thm:main}), we give a formula for the mixed volume of  \emph{partionable} binomial networks after showing that computing the mixed volume of these networks amounts to finding the volume of a single mixed cell that is the translate of a parallelotope (Theorem \ref{thm:MixedCellTheorem}). The subtlety here is that while the steady-state ideal is binomial for these networks, adding the conservation equations, which are linear equations in the species concentrations, can result in a steady-state system that is not binomial.  Our second key theorem (Theorem \ref{thrm:toriccycles}) concerns  identifying binomial networks, with a particular focus on cycle networks. In \cite{PDSC2012}, Perez Millan, Dickenstein, Shiu, and Conradi, give a sufficient condition on the complex-to-species rate matrix that guarantees a binomial steady-state ideal. In Theorem \ref{thrm:toriccycles}, we show that, for directed cycles, their condition is equivalent to a coloring condition on the underlying directed graph.  We showcase both theorems through the example of \emph{species-overlapping cycles}, giving a formula for the mixed volume of species overlapping cycles in Theorem \ref{thm:overlappingspecies}.  

This paper is organized as follows: in Section 2, we review chemical reaction networks, including networks with binomial steady-states, and mixed volumes of polynomial systems, including fine mixed subdivisions.  In Section 2, we also define PDSC networks, networks that satisfy a condition for binomiality of Perez Millan, Dickenstein, Shiu, and Conradi \cite{PDSC2012}. We call such a network a PDSC network and we show that in this case, it is straighforward to check if the steady-state system is equivalent to a square system. In Section 3, we give a formula for the mixed volume for partionable network binomial networks. In Section 4, we study cycles with binomial steady-states, giving a condition on directed cycles that guarantees toric steady-states.  We end Section 4 with an investigation of species-overlapping cycles and give a fomula for the mixed volume of these networks. We conclude the manuscript with a discussion in Section 5.   
\section{Background}

Here we review chemical reaction network theory with a particular focus on the existing literature regarding networks with binomial steady-state ideals. We also define the mixed volume of a polynomial system and review the method of computing mixed volumes via fine mixed subdivisions.

\subsection{Chemical reaction networks.}

To motivate the definition of a chemical reaction network, it is beneficial to follow an example of a chemical reaction network and the polynomial ODEs that arise from it. Consider a closed system consisting of three species $A$,$B$, and $C$ whose interactions are represented pictorially as in Figure \ref{fig:crn_example}.
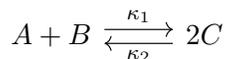
\begin{figure}[h]
    \centering
    \begin{tikzcd}
        A + B \ar[r,"\kappa_1",shift left] & 2C \ar[l,"\kappa_2",shift left]
    \end{tikzcd}
    \caption{A chemical reaction network involving two reactions and three species.}
    \label{fig:crn_example}
\end{figure}

The \textit{complexes} of the above network are $A+B$ and $2C$. \textit{Reactions} are denoted by labeled arrows between complexes, where the labels are positive real numbers called \textit{reaction rate constants} and can be thought to govern the rate of the reaction.
Interpret the reaction $$\begin{tikzcd}[column sep = small]
    A+B \ar[r,"\kappa_1"] & 2C
\end{tikzcd}$$ to mean $A$ and $B$ react to create two copies of $C$, hence the complex label $2C$. With the interest of tracking the amount, or \emph{concentration}, of $A$, $B$, and $C$ in the system, the net production of $A$, for instance, in the reaction $\begin{tikzcd}[column sep = small]
    A+B \ar[r,"\kappa_1"] & 2C
\end{tikzcd}$ is $-1$ since there is a loss of one copy of $A$ if this reaction fires once. On the other hand, the net production of $C$ in this same reaction is $2$ since two copies of $C$ are gained. These integer values are called \textit{stoichiometric coefficients} and they are determined by the structure of the complexes in each reaction.

Under the assumption of mass-action kinetics, a reaction outputs its products proportionally to the product of the concentrations of reacting species. With this assumption, the change in concentration is a function of the species concentrations where each reaction contributes a term to this function. If $x_A(t)$, $x_B(t)$, and $x_C(t)$ denote the concentrations of $A$, $B$, and $C$ respectively at a given time $t$, then for example the reaction $\begin{tikzcd}[column sep = small]A+B \ar[r,"\kappa_1"] & 2C
\end{tikzcd}$ contributes the term $-\kappa_1 x_Ax_B$ to the change in concentration of $A$ and $B$ and the term $2\kappa_1 x_Ax_B$ to the change in concentration of $C$. The associated differential equations are therefore polynomial in the species concentrations where each reaction contributes a monomial term whose coefficient is a product of a rate constant and a stoichiometric coefficient. Thus, the assumption of mass-action kinetics gives rise to the following polynomial ordinary differential equations:
\begin{align*}
    \frac{dx_A}{dt} &= -\kappa_1 x_Ax_B + \kappa_2 x_C^2\\
    \frac{dx_B}{dt} &= -\kappa_1 x_Ax_B + \kappa_2 x_C^2\\
    \frac{dx_C}{dt} &= 2\kappa_1 x_Ax_B - 2\kappa_2 x_C^2.
\end{align*}

In general, a chemical reaction network is a directed graph $G$ without loops whose vertices are the complexes of the network and whose edges are labeled by the reaction rate constants. Following the notation in \cite{PDSC2012}, let $s$ and $m$ denote the number of species and number of complexes, respectively. Each of the complexes $\bfy_1,\ldots,\bfy_m$ are formal linear sums of the species of the network, where the coefficients are nonnegative integers. We collect the coefficients into a matrix $Y = \begin{pmatrix}
y_{ij}
\end{pmatrix} \in \mathbb{Z}^{m \times s}$ and identify the $i$th row of this matrix with the complex $\bfy_i$.

Assuming mass action kinetics, every pair of species interact with equal probability and is independent of location. The rate of production for the reaction $\bfy_i \to \bfy_j$ is therefore proportional to a monomial in the concentrations of reacting species and this monomial is denoted by $\bfx^{\bfy_i} = \prod_{k=1}^s x_k^{y_{ik}}$ where $x_{k}$ denotes the concentration of the $k$th species. After choosing the proportional constants for every reaction, label the edge $\bfy_i \to \bfy_j$ with the parameter $\kappa_{ij} \in \mathbb{R}_{>0}$. Let $A_{{\boldsymbol\kappa}}$ denote the negative Laplacian of the network, namely, the $m \times m$ matrix whose row sums are zero and its $(i,j)$-entry is $\kappa_{ij}$ if $\bfy_i \to \bfy_j$ is an edge of the graph $G$. A main object of study is the matrix product $\Sigma = Y^t \cdot A_{\boldsymbol\kappa}^t$, called the \textit{complex-to-species rate matrix}. Every network defines the polynomial ODE system
\begin{equation}\label{eq:massactionsystem}
    f(\bfx) = \frac{d\bfx}{dt} = \Sigma\, \bfx^{Y^t},
\end{equation}
where the notation $\bfx^B$ for a matrix $B$ with $m$ columns $\bfb_i$ denotes the vector of monomials $(\bfx^{\bfb_1},\ldots,\bfx^{\bfb_m})^t$. 
The ideal generated by the polynomials in the above system is called the \textit{steady-state ideal} and is denoted $I_G$. A \textit{chemical reaction system} refers to the system of differential equations (\ref{eq:massactionsystem}) associated to a network after making a choice of rate constants ${\boldsymbol\kappa} = (\kappa_{ij})$ where $\kappa_{ij} \in \R_{>0}$ for each reaction $\bfy_i \to \bfy_j$. For instance, the matrices associated to the motivational example are 
$$\begin{array}{cc}
    Y^t = \begin{pmatrix}
        1 & 0\\
        1 & 0\\
        0 & 2
    \end{pmatrix}, & A_{\boldsymbol\kappa}^t = \begin{pmatrix}
        -\kappa_1 & \kappa_2\\
        \kappa_1 & -\kappa_2
    \end{pmatrix}, \\\\
    \Sigma = \begin{pmatrix}
        -\kappa_1 & \kappa_2\\
        -\kappa_1 & \kappa_2\\
        2\kappa_1 & -2\kappa_2
    \end{pmatrix},\text{ and} & \begin{pmatrix}
        \bfx^{\bfy_1}\\
        \bfx^{\bfy_2}
    \end{pmatrix} = \begin{pmatrix}
        x_Ax_B\\
        x_C
    \end{pmatrix}.
\end{array}$$


We remark that the chemical reaction system (\ref{eq:massactionsystem}) does not typically have full rank, and hence its vanishing locus is not zero-dimensional. In particular, there may be equations that are linearly dependent. This is easily seen by factoring the system (\ref{eq:massactionsystem}) in another way. It is sometimes factored as $f(\bfx) = N \diag({\boldsymbol\kappa}) \bfx^B$ where $\diag({\boldsymbol\kappa})$ is a diagonal matrix with diagonal entries given by ${\boldsymbol\kappa}$ and the columns of $B$ are $\bfy_i$ if $\bfy_i \to \bfy_j$ is a reaction. The matrix $N$ is called the \textit{stoichiometric matrix} and its columns are $\bfy_j - \bfy_i$ whenever $\bfy_i \to \bfy_j$ is a reaction. The column span of $N$ is known in the literature as the \textit{stoichiometric subspace}, it is a linear subspace of $\R^s$.

The redundancy of the equations in the system (\ref{eq:massactionsystem}) are due to the elements of the left kernel of $N$, i.e. if $W N = 0$ then $W f(\bfx) = 0$. Antidifferentiating gives rise to the linear equations $W\bfx = c$ called \textit{conservation laws}. The vectors $\bfw$ belonging to the left kernel of $N$ are referred to as \textit{conservation law vectors} and the set of all such vectors is a linear subspace of $\R^s$ is referred to as the \textit{linear space of conservation laws}.

If redundant equations in the system (\ref{eq:massactionsystem}) are replaced by a conservation law, the new system is full rank. For the chemical reaction network in Figure \ref{fig:crn_example}, we have
$$\begin{array}{cc}
   N = \begin{pmatrix}
       -1 & 1\\
       -1 & 1\\
       2 & -2\\
   \end{pmatrix}, \text{ and}  & W = \begin{pmatrix}
       1 & -1 & 0\\
       0 & 2 & 1
   \end{pmatrix} 
\end{array}$$
and the new system is 
\begin{align*}
    0 &= x_A - x_B - c_1\\
    0 &= 2x_B + x_C - c_2\\
    0 &= 2\kappa_1 x_Ax_B - 2\kappa_2 x_C^2.
\end{align*}


In the present work, we are interested in chemical reaction networks where the steady-state ideal is a binomial ideal. Occasionally the differential equations arising from mass action kinetics on the network are themselves binomial, as is the case in Example \ref{ex:Subdivision} in the next section. However it can also happen that these differential equations are not binomial, but ideal combinations of them are so that their steady-state ideal is generated by binomials. In this case, we say that the network has \emph{binomial steady-states}. The following condition, introduced by Perez Millan, Dickenstein, Shiu, and Conradi in \cite{PDSC2012}, is a sufficient condition for the network to have binomial steady-states.

\begin{condition}\label{cond:toricsteadystates}
For a chemical reaction system given by a network $G$ with $m$ complexes and reaction rate constants $\kappa_{ij}$, let $\Sigma$ denote its complex-to-species rate matrix, and set $d = \dim(\ker(\Sigma))$. We say that the chemical reaction system satisfies Condition \ref{cond:toricsteadystates}, if there exists a partition $I_1, \ldots, I_d$ of $\{1,\ldots,m\}$ and a basis $\bfb^1,\ldots,\bfb^d$ of $\ker(\Sigma)$ with $\supp(\bfb^i) = I_i$.
\end{condition}

We will refer to networks as \emph{PDSC networks} if they satisy Condition \ref{cond:toricsteadystates}, since PDSC are the initials of the authors of \cite{PDSC2012} where this condition was introduced. These authors prove the following result about PDSC networks.

\begin{theorem}[\cite{PDSC2012}, Theorem~3.3]
Let $G$ be a PDSC network as described in Condition \ref{cond:toricsteadystates}. Then the steady-state ideal $I_G$ is generated by the binomials of the form
\[
b^j_{j_1} \bfx^{\bfy_{j_2}} - b^j_{j_2} \bfx^{\bfy_{j_1}}
\]
for all $j_1, j_2 \in I_j$ and all $j \in [d]$.
\end{theorem}

Note that for a fixed $j \in [d]$, the binomials of the form $b^j_{j_1} \bfx^{\bfy_{j_2}} - b^j_{j_2} \bfx^{\bfy_{j_1}}$ for $j_1, j_2 \in I_j$ are generated by the $\# I_j - 1$ binomials
\[
b^j_{j'} \bfx^{\bfy_{j_2}} - b^j_{j_2} \bfx^{\bfy_{j'}}
\]
where $j'$ is fixed and $j_2 \in I_j \setminus \{j'\}$. This observation yields the following corollary.

\begin{corollary}\label{cor:PDSCgens}
Let $G$ be a PDSC network with $m$ complexes and with $\dim (\ker (\Sigma)) = d$. Then $I_G$ has a generating set consisting of $m-d$ binomials.
\end{corollary}

\subsection{Mixed volumes.}

In this section, we introduce the concept of the mixed volume of a set of polytopes and their applications for counting solutions to systems of polynomial equations. We first introduce the mixed volume for any set of integer polytopes. Then, we introduce the \emph{Newton polytope} of a polynomial and state the Bernstein-Khovanskii-Kouchnirenko 
(BKK) Theorem which bounds the number of solutions to a polynomial system in terms of the mixed volume of its Newton polytopes. We conclude the section with a discussion of the techniques used to compute the mixed volume.

Let $P_1, \dots, P_r$ be polytopes in $\R^r$. Their \emph{Minkowski sum} is the polytope in $\R^r$,
\[
P_1 + \dots + P_r := \{\bfv_1 + \dots + \bfv_r \mid \bfv_i \in P_i \text{ for all } i \}.
\]
Consider the volume of the Minkowski sum, $\mu_1 P_1 + \dots + \mu_r P_r$ for some positive $\mu_i$. In fact, this volume is a polynomial in the variables $\mu_1,\dots,\mu_r$ \cite[Theorem~5.1.7]{schneider1993}. 

\begin{definition}
    The \emph{mixed volume} of $P_1,\dots,P_r$, denoted $\mvol(P_1,\dots,P_r)$, is the coefficient of $\prod_{i=1}^r \mu_i$ in the polynomial $\vol(\mu_1 P_1 + \dots + \mu_r P_r)$.
\end{definition}

For background on mixed volumes, we refer the reader to \cite{betke1992} and \cite{schneider1993}. For their applications to solving sparse polynomial systems, \cite{bernstein1975} and \cite{huber1995} are good starting points.

Let $f \in \C[x_1^{\pm},\dots,x_r^{\pm}]$ be a Laurent polynomial with complex coefficients. 
Its \emph{support}, denoted $\supp(f)$, 
is the set of all $\bfy \in \Z^r$ such that $\bfx^\bfy$ has nonzero coefficient in $f$. 
The \emph{Newton polytope}, denoted $\Newt(f)$, 
is the convex hull of the support of $f$; that is,
\[
\Newt(f):= \conv\{\bfy \mid \bfy \in \supp(f)\}.
\]
This allows us to define the notion of the mixed  volume of a square system of Laurent polynomials.

\begin{definition}
    Let $F = \{f_1 = 0, \dots, f_r =0\}$ be a system of $r$ Laurent polynomials over $\C$ in $r$ variables. The \emph{mixed volume of} $F$, denoted $\mvol(F)$, is the mixed volume of the Newton polytopes of its polynomials; that is,
    \[
    \mvol(F) := \mvol(\Newt(f_1),\dots, \Newt(f_r)).
    \]
\end{definition}

The following theorem, known as the Bernstein-Khovanskii-Kouchnirenko (BKK) bound, relates the number of solutions of a square polynomial system to its mixed volume.

\begin{theorem}[BKK Bound \cite{bernstein1975,kouch1976}]
Let $p_1, \dots, p_r$ be a polynomial system in the Laurent polynomial ring $\C[x_1^{\pm 1},\dots,x_r^{\pm 1}]$. This system has at most
\[
\mvol(\Newt(p_1),\dots, \Newt(p_r))
\]
solutions in $(\C^*)^n$, with equality if the coefficients of the polynomials $p_i$ are sufficiently generic given their support. 
\end{theorem}

The coefficients of the polynomials arising from a chemical reaction network are not always generic since the reaction rates can appear as coefficients in more than one equation in the system. So the mixed volume of the mass-action system is simply an upper bound on the number of complex steady-states with no zero entries without the guarantee of equality in general.

One standard way to compute the mixed volume of a square polynomial system $f_1,\dots,f_r$ is to compute a fine mixed subdivision of the set of supports of each $f_i$. We introduce the relevant definitions and theorems for square polynomial systems following the notation of \cite{CLS2014} and \cite{huber1995}. Denote by $\Acal_i$ the support of the polynomial $f_i$, and let $\Acal = (\Acal_1,\dots,\Acal_r)$. A \emph{cell} of $\Acal$ is any tuple of the form $(C^1,\dots,C^r)$ where each $C^i$ is a non-empty subset of $\Acal_i$. For any cell $C = (C^1,\dots,C^r)$, we denote by $\conv( C)$ the Minkowski sum, $\sum_{i=1}^r \conv(C^i)$. The \emph{type} of $C$ is $\textsf{type(C)} := (\dim \conv (C^1), \dots, \dim \conv (C^r)).$

\begin{definition}
    A set $S = \{S_1,\dots,S_\ell \}$ consisting of cells of $\Acal$ is a \emph{subdivision} if
    \begin{itemize}
        \item $\dim \conv( S_i) = r$ for all $i \in [\ell]$,
        \item $\conv(S_i) \cap \conv(S_j)$ is a face of both $\conv(S_i)$ and $\conv(S_j)$ for all $i,j \in [\ell]$, and
        \item $\conv(\Acal) = \cup_{i=1}^\ell \conv(S_i)$.
    \end{itemize}
    For each cell in $S$, let $S_i = (S_i^{(1)},\dots,S_i^{(r)})$.
    This subdivision is \emph{mixed} if for each $S_j \in S$, we have
    \[
    \sum_{i=1}^r \dim \conv(S_j^{(i)}) = r.
    \]
    Finally, a subdivision is a \emph{fine mixed subdivision} if for each $S_j \in S$, we have that
    \[
    \sum_{i=1}^r (\#S_j^{(i)} -1) = r.
    \]
\end{definition}

This definition can be interpreted geometrically as giving a subdivision of the Minkowski sum of the Newton polytopes of each $f_i$. In particular, if $S = \{S_1, \dots, S_\ell\}$ is a fine mixed subdivision of $\mathcal{A}$, then the set of polytopes $\{ \conv(S_1) ,\dots, \conv(S_\ell)\}$ is a subdivision of $\sum_{i=1}^r \Newt(f_i)$. For this reason, we refer to a fine mixed subdivision of the tuple $\mathcal{A}$ of supports of the polynomials and of the Minkowski sum of their Newton polytopes interchangeably.  The following theorem relates the mixed volume of this polynomial system to the cells of this subdivision that are Minkowski sums of line segments. It follows directly from Theorem 2.4 of \cite{huber1995} or Proposition 12 of \cite{CLS2014} and the definition of the mixed volume of a system.

\begin{theorem}\label{thm:FMS}
Let $f_1, \dots, f_r$ be a square polynomial system in $\C[x_1,\dots,x_r]$ and let $\mathcal{A} = (\mathcal{A}_1, \dots, \mathcal{A}_r)$ where $\mathcal{A}_i$ consists of all exponent vectors in the support of $f_i$. Let $S = \{S_1, \dots, S_\ell\}$ be a fine mixed subdivision of $\mathcal{A}$. Then the mixed volume of this system is
\[
\sum_{\substack{S_i \in S \\ \textsf{type}(S_i) = (1,\dots,1)}} \vol(\conv(S_i)).
\]
\end{theorem}

\begin{example}[Species-overlapping Cycle]\label{ex:Subdivision} Consider the directed cycle with three complexes and three species given by:
\begin{center}
\begin{tikzpicture}[->,scale=1]
    \node (AB) at (90:1cm) {$A+B$};
    \node (A) at (-30:1cm) {$B+C$};
    \node (B) at (210:1cm) {$A+C$};
    
    \draw (55:1cm)  arc (55:-10:1cm) node[midway,right] {$\kappa_1$};
    \draw (-50:1cm) arc (-50:-130:1cm) node[midway,below] {$\kappa_2$};
    \draw (190:1cm) arc (190:125:1cm) node[midway,left] {$\kappa_3$};
    \end{tikzpicture}
\end{center}

We call this network a \emph{species-overlapping cycle} and characterize the mixed volumes of these types networks in Section 4. The system of ordinary differential equations arising from this network is
\begin{align*}
    f_A & = \kappa_2 x_B x_C - \kappa_1 x_A x_B \\
    f_B &= \kappa_3 x_A x_C - \kappa_2 x_B x_C \\
    f_C &= \kappa_1 x_A x_B - \kappa_3 x_A x_C.
\end{align*}
Note that this is a binomial system which is redundant as $f_A + f_B + f_C = 0$. It has a single conservation law which states that $x_A + x_B + x_C = c$ for some constant $c$. So we consider a fine mixed subdivision of the set $\mathcal{A} = \{\mathcal{A}_1, \mathcal{A}_2, \mathcal{A}_3 \}$  where
\begin{align*}
\mathcal{A}_1 &= \{\bfe_2 + \bfe_3, \bfe_1 + \bfe_2\}, \\
\mathcal{A}_2 &= \{ \bfe_2 + \bfe_3, \bfe_1 + \bfe_2 \}, \text{ and} \\
\mathcal{A}_3 & = \{ \mathbf{0}, \bfe_1, \bfe_2, \bfe_3 \}.
\end{align*}

Consider the collection of cells,
\begin{align*}
    S_1 & = \bigl\{ \mathcal{A}_1, \mathcal{A}_2, \{ \mathbf{0}, \bfe_1 \} \bigr\} \\
    S_2 & = \bigl\{ \{\bfe_1 + \bfe_2 \}, \{\bfe_2 + \bfe_3\}, \mathcal{A}_3 \bigr\} \\
    S_3 & = \bigl\{ \{ \bfe_1 + \bfe_2 \}, \mathcal{A}_2, \{ \mathbf{0}, \bfe_1, \bfe_3 \} \bigr\} \\
    S_4 & = \bigl\{ \mathcal{A}_1, \{ \bfe_2 + \bfe_3\}, \{ \mathbf{0}, \bfe_2, \bfe_3 \} \bigr\}. 
\end{align*}
One can check that this forms a fine mixed subdivision of $\Acal$ as depicted in Figure \ref{fig:Subdivision}. Note that $S_1$ is the only cell of type $(1,1,1)$ and its volume is 1, so by Theorem \ref{thm:FMS}, the mixed volume of this system is 1.
\end{example}

\begin{figure}
    \centering
    \includegraphics[scale=.25]{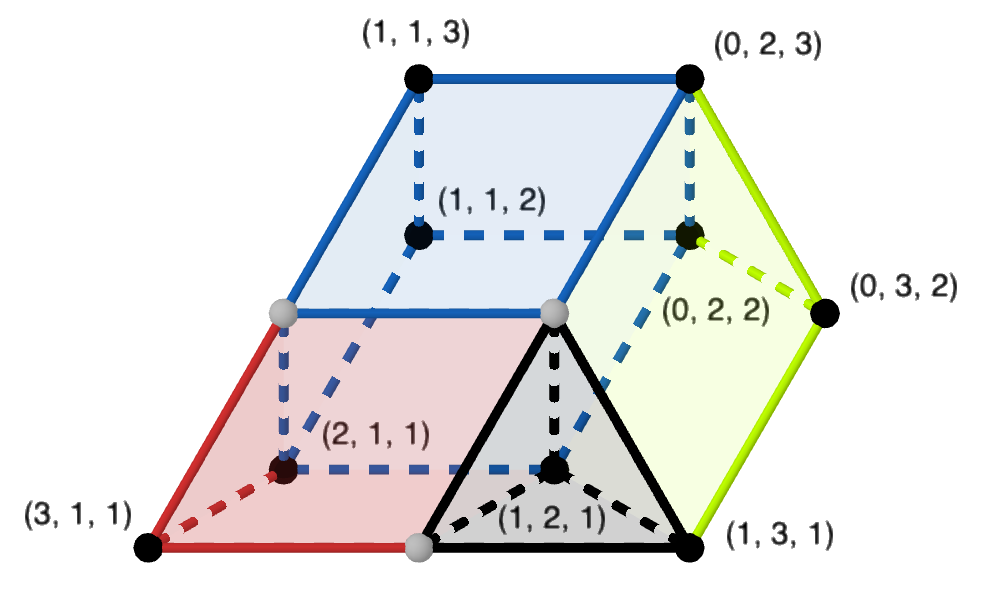}
    \caption{The fine mixed subdivision described in Example \ref{ex:Subdivision}. The cells are $S_1$ in blue, $S_2$ in black, $S_3$ in red and $S_4$ in green.}
    \label{fig:Subdivision}
\end{figure}

We note here that the mixed volume is not an intrinsic property of a chemical reaction network. In fact, one can obtain different mixed volumes from different choices of generators of the steady-state ideal and/or conservation laws. We illustrate an example of this below.

\begin{example}\label{ex:GenSet}
Consider the network pictured in Figure \ref{fig:GenSet}. The system of ordinary differential equations arising from this system is
\begin{align*}
    f_A &= -\kappa_1 x_A \\
    f_B &= \kappa_1 x_A - \kappa_2 x_B x_C + \kappa_3 x_C^2 \\
    f_C &= \kappa_1 x_A + \kappa_2 x_B x_C -  \kappa_3 x_C^2.
\end{align*}
It has a single conservation law, which requires that $w := 2 x_A + x_B + x_C - c = 0$ for some generic constant $c$. Note that the steady-state ideal $I_G = \langle f_A, f_B \rangle = \langle f_B, f_C \rangle$. The system $\{f_A, f_C, w\}=0$ has mixed volume $0$ since $f_A$ is a monomial. Indeed, any system including a monomial has mixed volume zero since there can be no cell of type $(1,\dots,1)$ in any fine mixed subdivision of the Minkowski sum of the Newton polytopes in the system. However,  the system $\{f_B, f_C, w \}$ contains no monomials, and we can compute using the PHCpack package for Macaulay2 \cite{M2,gross2013interfacing, verschelde1999algorithm} that its mixed volume is 2.
\end{example}

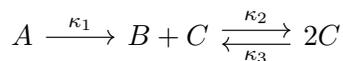
\begin{figure}[h]
    \centering
    \begin{tikzcd}
        A \ar[r,"\kappa_1"] & B + C \ar[r,"\kappa_2",shift left] & 2C \ar[l, "\kappa_3", shift left]
    \end{tikzcd}
    \caption{A network where the mixed volume depends on the choice of generating set.}
    \label{fig:GenSet}
\end{figure}

\subsection{Squareness of PDSC Networks}

In order to compute the mixed volume of a system of steady-state equations and conservation laws, that system must be square. It is not necessarily straightforward to check whether the steady-state ideal has a binomial generating set that makes the system augmented by conservation laws square; however, in the case of PDSC networks, we obtain a sufficient condition as follows.

Let $G$ be a network. A \emph{linkage class} of $G$ is a connected component of its underlying directed graph. Let $\ell$ denote the number of linkage classes of $G$. Following \cite{gunawardena2003}, we define its \emph{deficiency} by
\[
\delta := \dim (\ker Y^t \cap \im A_\kappa^t).
\]
Note that this is equal to $\dim \ker Y^t A_\kappa^t - \dim \ker A^t_\kappa.$ There are other definitions of deficiency in the literature which are not always equivalent. In particular, it is often defined as $m - \ell - \rank(N)$.  Proposition 5.1 of \cite{gunawardena2003} states that if every linkage class has exactly one terminal strong linkage class, then these definitions coincide; that is, in this case, $\rank(N) = m - \ell - \dim \ker Y^t A_\kappa^t + \dim \ker A^t_\kappa$.

\begin{proposition}
Let $G$ be a PDSC network such that every linkage class contains exactly one terminal strong linkage class. Then the system consisting of the binomial generators guaranteed by \ref{cor:PDSCgens} augmented by a basis of conservation laws is a square system.
\end{proposition}

\begin{proof}
The nullity of the weighted Laplacian $A_\kappa$ is equal to the number of terminal strong linkage classes \cite{feinberg1977chemical}, and hence equal to the number of linkage classes. Thus by Proposition 5.1 of \cite{gunawardena2003}, we have
\begin{align*}
    \rank(N) &= m - \ell - \dim \ker Y^t A_\kappa^t + \dim \ker A^t_\kappa \\
    &= m - \ell - (m - \rank Y^t A_\kappa^t) + \ell \\
    & = \rank Y^t A_\kappa^t.
\end{align*}
Thus the nullity of $N$ and $Y^t A_\kappa^t$ are both some fixed $d$. Let $\{\bfw_1, \dots, \bfw_d\}$ be a basis for $\ker N$. Let $f_1, \dots, f_{s-d}$ be the generating set for $I_G$ guaranteed by Corollary \ref{cor:PDSCgens}. Then the system obtained by augmenting $f_1, \dots, f_{s-d}$ with the conservation equations associated to $\bfw_1,\dots,\bfw_d$ is a square system, as needed.
\end{proof}



\section{Mixed Volumes of Partitionable Binomial Networks}

Now we will focus on a specific class of binomial networks, particularly \emph{partitionable networks}, and we show that, for these networks, computing the mixed volume reduces to computing the volume of a single mixed cell. We further show that one need not actually find such a mixed cell -- in fact, in these cases, the mixed volume can be computed without computing a fine mixed subdivision. 
In order to define paritionable networks, we need the following algebraic notion of multihomogeneity.

\begin{definition}
Let $I$ be an ideal in $\mathbbm{k}[x_1,\dots,x_s]$. Let $\bfw_1,\dots, \bfw_k \in \Z^s$ be integer weight vectors. Then $I$ is \emph{multihomogeneous} with respect to the \emph{multigrading} specified by $\bfw_1, \dots, \bfw_k$ if it has a generating set $f_1,\dots,f_{s-k}$ such that for each $f_i = \sum_{\bfa \in \mathcal A_i} \beta_{\bfa} \bfx^\bfa$, we have $\bfa \cdot \bfw_j = \bfb \cdot \bfw_j$ for all $j \in [k]$ and $\bfa, \bfb \in \mathcal A_i$. 
\end{definition}

We note that multihomogeneity with respect to $\bfw_1,\dots,\bfw_k$ is a property of $\textsf{span}\{\bfw_1,\dots,\bfw_k\}$ and does not depend on the choice of spanning set of this vector space. Indeed, an ideal is multihomogenous with respect to $\bfw_1,\dots,\bfw_k$ if and only if it is homogeneous with respect to the weight order specified by any $\bfw$ in their span. 

We can now define a \emph{partitionable network}, where the structure of the conservation laws leads to very nice geometry on the level of mixed volumes.

\begin{definition}\label{def:partitionable}
A network $G$ is \emph{partitionable} if 
\begin{enumerate}
\item  there are $0/1$ vectors $\bfw_1,\dots,\bfw_k$ with disjoint support such that the linear space of all conservation laws of $G$ is equal to $\spn{\{\bfw_1,\dots,\bfw_k\} }$ and
\item the steady-state ideal is multihomogeneous with respect to the multigrading induced by the conservation laws. 
\end{enumerate}
\end{definition}

\begin{rmk} The multihomogeneity condition for a network to be partitionable has a nice geometric interpretation. This condition is equivalent to the affine hull of the newton polytope $\Newt(f_i)$ being parallel to the stoichiometric subspace for each generator $f_i$ of the steady-state ideal.
\end{rmk}

Observe that the first of the conditions in Definition \ref{def:partitionable} is the more restrictive one; in the proof of Theorem \ref{thm:MixedCellTheorem}, it places significant restrictions on the form of the mixed cells that can appear in a fine mixed subdivision. The second condition is more mild. In particular, it is satisfied if the evaluations of a conservation law on each complex are equal. Noteably, the undirected graph underlying the network is connected; in this case, the network is referred to as \emph{weakly connected}.

\begin{proposition}
If $G$ is weakly connected,  then its steady-state ideal $I_G$ is multihomogeneous with respect to the multigrading induced by the conservation law vectors.
\end{proposition}

\begin{proof}
Let $G$ be a weakly connected network with complexes $\bfy_1,\dots,\bfy_m$ and conservation law vectors $\bfw_1,\dots,\bfw_k$. Then by definition of a conservation law, for each reaction $\bfy_i \rightarrow \bfy_j$ and each conservation law vector $\bfw_\ell$, we have $\bfy_i \cdot \bfw_\ell = \bfy_j \cdot \bfw_\ell$. Moreover, since $G$ is weakly connected, there is an undirected path between each pair of complexes. Thus this equality holds for any pair of complexes $\bfy_i$ and $\bfy_j$ and any conservation law vector $\bfw_\ell$. Every term of the steady-state equations $d\dot{\bfx}/dt = 0$ that generate $I_G$ is of the form $\bfx^{\bfy_i}$ for some complex $\bfy_i$. Thus the steady-state equations form a generating set of $I_G$ that is multihomogeneous with respect to the multigrading specified by the conservation law vectors.
\end{proof}

\begin{example}
    As a non-example, the Edelstein network
\[
\begin{array}{c}
    \begin{tikzcd}
        A \ar[r,"\kappa_1",shift left] & 2A \ar[l,"\kappa_2",shift left]
    \end{tikzcd} \\
    \begin{tikzcd}
        A + B \ar[r,"\kappa_3",shift left] & C \ar[r,"\kappa_5",shift left] \ar[l,"\kappa_4",shift left] & B \ar[l,"\kappa_6",shift left]
    \end{tikzcd}
\end{array}
\]
is not partitionable. It has one conservation law vector $\bfw = (0,1,1)$ and if we consider the two exponent vectors $\bfa = (2,0,0)$ and $\bfb = (1,1,0)$ of the polynomial $f_A(\bfx) = \kappa_1x_A - \kappa_2x_A^2 -\kappa_3x_Ax_B + \kappa_4x_C$, we compute $\bfa \cdot \bfw = 0$ while $\bfb \cdot \bfw = 1$.
\end{example}

The main result of this section is the following characterization of the mixed volume of a partitionable binomial reaction network. In particular, we find that any fine mixed subdivision of the Minkowski sum of the Newton polytopes of such a network has at most one cell of type $(1,\dots,1)$. This allows us to easily compute the mixed volume of such a network, since if a type $(1,\dots,1)$ cell exists, then the mixed volume is the volume of this single cell. If no such mixed cell exists, then the mixed volume is zero.

\begin{theorem}\label{thm:MixedCellTheorem}
Let $G$ be a paritionable binomial network with $s$ species and $k$ conservation law vectors $\bfw_1,\dots,\bfw_k$ with disjoint support. Suppose that its steady-state ideal has a binomial generating set $f_1,\dots,f_{s-k}$ with exactly $s-k$ elements. Then any fine mixed subdivision of $\sum_{i=1}^{s-k} \Newt(f_i) + \sum_{j=1}^k \Newt(\bfw_j \bfx - c_j)$ has at most one cell of type $(1,\dots,1)$. This cell, if it exists, is a translate of some parallelotope of the form
\[
\sum_{i=1}^{s-k} \Newt(f_i) + \sum_{j=1}^k \conv(\mathbf{0}, \bfe_{\alpha_j}),
\]
where $\alpha_j$ is in the support of $\bfw_j$ for all $j$.
\end{theorem}

Before we prove Theorem \ref{thm:MixedCellTheorem}, we require the following well-known proposition regarding the Minkowski sums of a polytope with two different edges.

\begin{proposition}\label{prop:InteriorIntersection}
Let $P$ and $Q$ be $d$-dimensional polytopes in $\R^d$ that share a facet $F$. Suppose further that $F$ is the face of both that maximizes the same linear functional $\bfa$. Then $P$ and $Q$ intersect on their interiors.
\end{proposition}

\begin{proof}
We have that $F = \{\bfx \in P \mid \bfa \cdot \bfx = b \} =  \{\bfx \in Q \mid \bfa \cdot \bfx = b \}$ and that $P$ and $Q$ are contained in the closed halfspace $\{ \bfx \mid \bfa \bfx \leq b \}$. Moreover, $F$ is a facet of both polytopes. So we may write minimal H-representations
\[
P = \{ \bfx \mid A \bfx \leq \bfb \} \qquad \text{ and } \qquad  Q = \{ \bfx \mid C \bfx \leq \bfd \}
\]
where the first rows of $A$ and $C$, $\bfa_1$ and $\bfc_1$ respectively, are both equal to $\bfa$ and the first entries of $\bfb$ and $\bfd$ are both $b$. Since $F$ is a facet, we may further assume that for all $i >1$ and all $\bfx \in F$, $\bfa_i \cdot \bfx < b_i$ and $\bfc_i \cdot \bfx < d_i$.

Let $\bfz \in \relint(F)$. We construct an $\epsilon > 0$ such that $\bfz - \epsilon \bfa^t \in \int(P) \cap \int(Q)$. Consider the rows $\bfa_i, \bfc_j$ of $A$ and $C$ respectively for $i,j > 1$. Let $\epsilon > 0$ be such that $-\epsilon \bfa_i \cdot \bfa^t < b_i - \bfa_i \cdot \bfz$ and $-\epsilon \bfc_j \cdot \bfa^t < d_j - \bfc_j \cdot \bfz$ for all $i, j > 1$. Such an $\epsilon$ exists since $b_i - \bfa_i \cdot \bfz$ and $d_j - \bfc_j \cdot \bfz$ are strictly positive for all $i, j > 1$. Then 
\[
\bfz - \epsilon \bfa^t \in \{\bfx \mid A \bfx < \bfb \} \cap \{\bfx \mid C \bfx < \bfd \},
\]
which is exactly the intersection of the interiors of $P$ and $Q$, as needed.
\end{proof}

\begin{proof}[Proof of Theorem \ref{thm:MixedCellTheorem}]
Since the mixed volume is translation invariant, we replace each $\Newt(f_i)$ with its translation to the origin, $P_i := \conv(\mathbf 0, \bfy_1^{(i)} - \bfy_2^{(i)})$, where $f_i = \bfx^{\bfy_1^{(i)}} - \bfx^{\bfy_2^{(i)}}$.
Let $P = \sum_{i=1}^{s-k} P_i + \sum_{j=1}^k \Newt(\bfw_j \bfx - c_j)$.
Each polytope $P_i$ is a line segment since $f_i$ is a binomial. Thus, for any fine mixed subdivision, a mixed cell of type $(1,\dots,1)$ must have each $P_i$ as a summand. 

Further, note that there are two types of edges of each simplex $W_j := \Newt(\bfw_j \bfx - c_j)$; they are of the form $\conv(\mathbf 0, \bfe_{\alpha})$ or $\conv(\bfe_{\alpha}, \bfe_{\beta})$ where $\alpha$ and $\beta$ are in the support of $\bfw_j$. For all $j \in [k]$, if $\alpha$ belongs to the support of $\bfw_j$, then since $G$ is partitionable, $\bfe_\alpha$ belongs to the linear space
\begin{equation*} \displaystyle
\bigcap_{\substack{h=1 \\ h \neq j}}^k \{ \bfp \mid \bfw_h \bfp = 0\}.
\end{equation*}
Moreover, since $G$ is partitionable, the steady-state ideal $I_G$ is multihomogeneous with respect to the multigrading induced by the conservation laws. Since $I_G$ contains no monomials, this implies that each $f_i$ is multihomogeneous with respect to this multigrading as well. Thus $\bfw_h \cdot (\bfy_1^{(i)} - \bfy_2^{(i)}) = 0$ for all $h \in [k]$ and $i \in [s-k]$. 

Let $Q$ be a type $(1,\dots,1)$ mixed cell of a fine mixed subdivision of $P$. Then $Q = \sum_{i=1}^{s-k} P_i + \sum_{j=1}^k E_j$ where each $E_j$ is an edge of $W_j$.
For the sake of contradiction, suppose that $E_j = \conv(\bfe_\alpha, \bfe_\beta)$ for some $j$. 
Then $\sum_{i=1}^{s-k}P_i + E_j$ lies in the codimension $k$ affine linear space
\[
\{\bfp \mid \bfw_j \bfp = 1\} \cap \bigcap_{\substack{h=1 \\ h \neq j}}^k \{ \bfp \mid \bfw_h \bfp = 0\}.
\]
So it has dimension less than or equal to $s-k$. Thus $Q$ has dimension less than or equal to $s-1$, which contradicts that it is a maximal cell of a fine mixed subdivision.

Thus all mixed cells of type $(1,\dots,1)$ are of the form $\sum_{i=1}^{s-k} P_i + \sum_{j=1}^k \conv(\mathbf 0, \bfe_{\alpha_j})$ where $\alpha_j$ belongs to the support of $\bfw_j$. Let $E_{\alpha_j}$ denote $\conv(\mathbf 0, \bfe_{\alpha_j})$.

Now suppose that $Q_1$ and $Q_2$ are distinct mixed cells of this form. They must differ by at least one summand corresponding to edges of some $W_j$. Without loss of generality, suppose that this is $\bfw_1$ and that the summand associated to $\bfw_1$ in $Q_1$ is $E_1$ and the summand associated to $\bfw_1$ in $Q_2$ is $E_2$. Consider the face of $P$ that minimizes the linear functionals $\bfw_j \bfp$ for all $j=2,\dots,k$. This face is $F = \sum_{i=1}^{s-k} P_i + W_1$. The fine mixed subdivision $\mathcal S$ of $P$ restricts to a subdivision $\mathcal{S}'$ of this face via intersection.
Moreover, we have $Q_1 \cap F = \sum_{i=1}^{s-k} P_i + E_1$ and $Q_2 \cap F = \sum_{i=1}^{s-k} P_i + E_2$. 

Now consider these polytopes in the ambient linear space, $\{\bfp \mid \bfw_j \bfp = 0, j=2,\dots,k \}$. The face $F$ is contained in the hyperplane $\{ \bfw_1 \bfp = 0 \}$ and $E_1$ and $E_2$ both lie in the positive halfspace defined by $\bfw_1 \bfp \geq 0$. So by Proposition \ref{prop:InteriorIntersection}, we have that $\relint(\sum_{i=1}^k P_i + E_1) \cap \relint(\sum_{i=1}^k P_i + E_2)$ is nonempty. Moreover, these two polytopes are not equal. So they cannot belong to the same subdivision of $F$. Hence, $Q_1$ and $Q_2$ cannot belong to the same subdivision of $P$. Thus a fine mixed subdivision of $P$ has at most one mixed cell of type $(1,\dots,1)$ and it has the desired form if it exists.
\end{proof}

Consider a partitionable binomial reaction network as in the statement of Theorem \ref{thm:MixedCellTheorem}. We have shown that any fine mixed subdivision of the Minkowski sum of its Newton polytopes has at most one mixed cell of type $(1,\dots,1)$ and described the form of this cell if it exists. Let $\Pi$ denote this parallelotope. If one knows the edges of each Newton polytope that are its Minkowski summands, then its volume can be computed as the determinant of a matrix.

\begin{lemma}\label{lem:PiVolume}
Let $G$ be as in the statement of Theorem \ref{thm:MixedCellTheorem} and suppose that 
\[
\Pi = \sum_{i=1}^{s-k} \Newt(f_i) + \sum_{j=1}^k \conv(\mathbf{0}, \bfe_{\alpha_j})
\]
be the unique type $(1,\dots,1)$ cell of a fine mixed subdivision of the Newton polytopes where each $f_i= \bfx^{\bfy^{(i)}_1} - \bfx^{\bfy^{(i)}_2}$ and where $\alpha_j$ is in the support of $\bfw_j$. Then the mixed volume of the steady-state system $f_1,\dots,f_{s-k}$ augmented by the partitionable conservation laws  the absolute value of the determinant of the $s \times s$ matrix with columns $\bfe_{\alpha_j}$ for $j \in [k]$ and $\bfy^{(i)}_1 - \bfy^{(i)}_2$ for $i \in [s-k]$
\end{lemma}

\begin{proof}
Suppose that this mixed volume is non-zero.
The volume of a polytope is invariant under translation.
To compute the volume of the parallelotope $\Pi$, we translate it to the origin by replacing the edge $\Newt(f_i) = \conv(\bfy_1^{(i)}, \bfy_2^{(i)})$ with $\conv(\mathbf{0}, \bfy_1^{(i)} - \bfy_2^{(i)})$. Then this determinant is the standard formula for the normalized volume of such a parallelotope. By Theorem \ref{thm:MixedCellTheorem}, $\Pi$ is the only mixed cell of type $(1,\dots,1)$ in a fine mixed subdivision of $P$. Thus by Theorem \ref{thm:FMS}, the mixed volume of $G$ is the volume of $\Pi$.
\end{proof}

We conclude this discussion by noting that this determinant does not depend on the choice of the coordinates $\alpha_j$ in the support of $\bfw_j$. So in fact, one can compute the mixed volume of a partitionable binomial reaction network via a simple determinant calculation without computing a fine mixed subdivision. In order to prove this, we state the following lemma.

\begin{lemma}\label{lem:Det}
Let $\bfr_1, \dots, \bfr_{k+1}$ be $s$-dimensional row vectors that sum to the zero vector. Let $\bfq_1, \dots, \bfq_{s-k}$ be arbitrary $s$-dimensional row vectors. Let $R_i$ denote the $s \times s$ matrix with rows $\bfr_1,\dots,\bfr_{k+1},\bfq_{1},\dots,\bfq_{s-k}$ with $\bfr_i$ excluded. Then $\det(R_i) = (-1)^{i-j} \det(R_j)$ for any $i, j \in [k+1].$
\end{lemma}

\begin{proof}
    For any $i,j$, we have 
    \[\bfr_i = -\sum_{\substack{\ell=1 \\ \ell \neq i}}^{k+1} \bfr_\ell.
    \]
    Replacing $\bfr_i$ with this expression in $R_j$ and expanding using multilinearity and the alternating property of the determinant yields that $\det(R_j) = - \det(R_j^{(i)})$, where $R_j^{(i)}$ is obtained from $R_j$ by replacing $\bfr_j$ with $\bfr_i$. Then using $i-j-1$ adjacent row swaps to put $\bfr_i$ in the $i$th position yields $R_i$. So $\det(R_j) = (-1)^{i-j} \det(R_i)$, as needed.
\end{proof}

\begin{theorem}\label{thm:main}
Let $G$ be a partitionable binomial reaction network with partitionable conservation laws $\bfw_1,\dots,\bfw_k$ and exactly $s-k$ defining binomials $f_i$ supported on exponent vectors $\bfy_1^{(i)}$ and $\bfy_2^{(i)}$. Then the mixed volume of the steady-state system $f_1,\dots,f_{s-k}$ augmented by the partitionable conservation laws is either $0$ or the absolute value of the determinant of \emph{any} matrix with columns $\bfy_1^{(i)} - \bfy_2^{(i)}$ for all $i \in [s-k]$ and $\bfe_{\alpha_j}$ for $\alpha_j \in \supp(\bfw_j)$ for each $j \in [k]$.
\end{theorem}

\begin{proof}
Suppose that the system $f_i(\bfx) = 0$ for $i \in [s-k]$ and $\bfw_j \bfx = c_j$ for $j \in [k]$ has nonzero mixed volume. Then any fine mixed subdivision of $p$ has a type $(1,\dots,1)$ cell, and by Theorem \ref{thm:MixedCellTheorem}, this cell is unique and of the form
\[
\Pi = \sum_{i=1}^{s-k} \Newt(f_i) + \sum_{j=1}^k \conv(\mathbf{0}, \bfe_{\alpha_j}).
\]
Let $M_{\alpha}$ denote the matrix with columns $\bfy_1^{(i)} - \bfy_2^{(i)}$ for all $i \in [s-k]$ and $\bfe_{\alpha_j}$ for $j \in [k]$. Then by Lemma \ref{lem:PiVolume}, the mixed volume of $G$ is equal to $\pm \det(M_\alpha)$.

It remains to show that if we pick $\beta_j \in \supp(\bfw_j)$ for each $j$, the corresponding matrix has the same determinant up to absolute value; that is, that $\det M_\alpha = \pm \det M_\beta$. Consider the $s \times (s-k)$ matrix $M$ with colums $\bfy^{(i)}_1 - \bfy^{(i)}_2$. Let $\bfr_1,\dots,\bfr_s$ be its rows. Note that for each $j \in [k]$, we have
\[
\sum_{\ell \in \supp(\bfw_j)} \bfr_\ell = 0.
\]
Let $M_{\alpha}'$ denote the matrix obtained by deleting rows $\bfr_{\alpha_1},\dots,\bfr_{\alpha_s}$ from $M$, and similarly for $M'_\beta$. Then by repeatedly applying Laplace expansion along the columns of the form $\bfe_{\alpha_j}$, we see that $\det M_\alpha = \pm \det M'_\alpha$, and similarly, that $\det M_\beta = \pm \det M'_\beta$. Moreover, by applying Lemma \ref{lem:Det} $s$ times to the block of rows $\{ \bfr_{\ell} \mid \ell \in \supp(\bfw_j) \}$ at the $j$th step of the Laplace expansion, we see that
\[
\det(M'_\alpha) = \pm \det(M'_\beta),
\]
as needed.
\end{proof}

\section{Cycles with Binomial Steady-States}

In this section, we investigate the directed cycles, or \emph{cycle networks}, that satisfy the PDSC Condition. We give a characterization of these cycles in terms of edge colorings of the cycle. Then we apply the results of Section 3 to some examples of cycles with binomial steady-states and compute their mixed volumes. 

Let $G$ be a cycle network with $m$ complexes, that is, defined by the reactions $\begin{tikzcd}[column sep = small] \bfy_i \ar[r,"\kappa_i"] & \bfy_{i+1} \end{tikzcd}$ where the indices are taken modulo the set $[m] = \{1,2,\ldots,m\}$. Let $d = \dim \ker \Sigma$ where $\Sigma = Y^t A_{\boldsymbol\kappa}^t$. Since we are fixing the structure of the reaction graph $G$, the Laplacian matrix has the form $$A_{\boldsymbol\kappa}^t = \begin{pmatrix}
-\kappa_1 & & & & & \kappa_m\\
\kappa_1 & -\kappa_2 & & & & \\
& \kappa_2 & -\kappa_3 & & & \\
& & & \ddots & & \\
& & & & -\kappa_{m-1} & \\
& & & & \kappa_{m-1} & -\kappa_m
\end{pmatrix}$$
which has nontrivial kernel; indeed, it contains the non-zero vector $\bfx_{\boldsymbol\kappa} = (\kappa_1^{-1}, \ldots, \kappa_m^{-1})^t$. Thus, for cycle networks, the dimension $d$ of the kernel of $\Sigma = Y^t A_{\boldsymbol\kappa}^t$ is always at least 1.

Given a coloring of the edges of $G$, $\lambda: E(G) \rightarrow C$, and a ``color" $\ell \in C$, the subgraph of $G$ induced by all edges of color $\ell$, denoted $G[\ell]$ is a disjoint union of directed paths if $|C| > 1$ and $G$ if $|C| = 1$. Let $H(\ell)$ denote the set of all source vertices of $G[\ell]$ and let $T(\ell)$ denote the set of all sink vertices of $G[\ell]$. Note that in the case $|C| = 1$, the sets $H(\ell)$ and $T(\ell)$ are both empty. Given a subset $S$ of the complexes of $G$, we shall write $\1_S$ to denote the $m$-dimensional indicator vector for $S$.

\begin{theorem}\label{thrm:toriccycles}
Let $G$ be a directed cycle. Then $G$ is a PDSC network if and only if there exists a surjective coloring $\lambda:E(G) \rightarrow [d]$ such that for all $\ell \in [d]$,
\[
\sum_{\bfy \in H(\ell)} \bfy = \sum_{\bfy \in T(\ell)} \bfy.
\]
\end{theorem}

\begin{proof}
Let $\lambda: E(G) \rightarrow [d]$ be a surjective coloring of the edges of $G$ such that for all $\ell \in [d]$, $\sum_{\bfy \in H(\ell)} \bfy = \sum_{\bfy \in T(\ell)} \bfy$.  Then the difference of indicator vectors $\1_{H(\ell)} - \1_{T(\ell)}$ belongs to $\ker Y^t$ for all $\ell \in [d]$. This is equal to the image of the vector $\bfb^{\ell}$ under $A_{\boldsymbol\kappa}^t$ where $\bfb^{\ell}$ is defined by
\[
b^{\ell}_i = \begin{cases}
\kappa_i^{-1} & \text{ if } \lambda(\bfy_i \rightarrow \bfy_{i+1}) = \ell \\
0 & \text{ otherwise.}
\end{cases}
\]

Indeed, if $\bfy_i$ is the source of a path in $G[\ell]$, then $b^{\ell}_i = \kappa_i^{-1}$ and $b^{\ell}_{i-1} = 0$. So the $i$th entry of $A_{\boldsymbol\kappa}^t \bfb^{\ell}$ is $-1$. Similarly, if $\bfy_i$ is the sink of a path in $G[\ell]$, then $b_i^{\ell} = 0$ and $b_{i-1}^\ell = \kappa_{i-1}^{-1}$. So the $i$th entry of $A_{\boldsymbol\kappa}^t \bfb^{\ell}$ is $1$. If $\bfy_i$ is an interior node on a path in $G[\ell]$, then $b_i^{\ell} = \kappa_i^{-1}$ and $b_{i-1}^\ell = \kappa_{i-1}^{-1}$, so that the $i$th entry of $A_{\boldsymbol\kappa}^t \bfb^{\ell}$ is $0$. Finally if $\bfy_i$ does not belong to $G[\ell]$, then $\bfy_{i-1}$ either is also not in $G[\ell]$ or is a sink of a path in $G[\ell]$. Hence we have $b^\ell_i = b^\ell_{i-1} = 0$, and the $i$th entry of $A_{\boldsymbol\kappa}^t \bfb^\ell$ is 0.

The vectors $\bfb^1,\dots,\bfb^d$ have disjoint support since each complex has exactly one outgoing end. Thus they are linearly independent. Moreover, they form a basis for $\ker \Sigma$ as they comprise $d$ distinct vectors. Thus $G$ satisfies Condition \ref{cond:toricsteadystates}.

Now suppose that $G$ satisfies Condition \ref{cond:toricsteadystates} and let $\bfb^1,\dots,\bfb^d$ be a basis for $\ker \Sigma$ with disjoint support. In particular, we know that $\bfx_{\boldsymbol\kappa} = (\kappa_1^{-1},\dots,\kappa_m^{-1})^t$ is in $\ker \Sigma$ as it belongs to $\ker A_{\boldsymbol\kappa}^t$. So it is in the span of $\bfb^1,\dots, \bfb^d$. Thus, after rescaling each $\bfb^\ell$, we have that if $j \in \textsf{supp}(\bfb^\ell)$, then $b_j^{\ell} = \kappa_j^{-1}$.

Color the edges of $G$ by letting the edge $\bfy_i \rightarrow \bfy_{i+1}$ have color $\ell$ if and only if $i \in \textsf{supp}(\bfb^\ell)$. Then $A_{\boldsymbol\kappa}^t \bfb^\ell = \1_{H(\ell)} - \1_{T(\ell)}$. Since $\bfb^\ell \in \ker\Sigma$, we must have that $\1_{H(\ell)} - \1_{T(\ell)} \in \ker Y^t$. Hence we have $\sum_{\bfy \in H(\ell)} \bfy = \sum_{\bfy \in T(\ell)} \bfy$, as needed.
\end{proof}

The above proof uncovers another key fact about PDSC cycle networks. In particular, when the reaction rates $\kappa_{i}$ are positive, these networks trivially satisfy another condition from \cite{PDSC2012}, which we restate below.

\begin{condition}[\cite{PDSC2012},Condition 3.4]\label{cond:positiveSS}
    Consider a chemical reaction system given by the PDSC network $G$ with $m$ complexes and reaction rate constants $\kappa_{ij}$. There is a partition $I_1, \ldots, I_d$ of $[m]$ and a basis $\bfb^1, \ldots, \bfb^d$ of $\ker \Sigma$ with $\supp \bfb^i = I_i$. We say that the chemical reaction system additionally satisfies Condition \ref{cond:positiveSS} if for all $j \in [m]$, the nonzero entries of $\bfb^j$ have the same sign, that is, if 
    \[
    \text{sign}(b^j_{j_1}) = \text{sign}(b^j_{j_2}) \quad \text{ for all } j_1,j_2 \in I_j, \text{ for all } 1 \leq j \leq d.
    \]
\end{condition}

Theorem 3.8 of \cite{PDSC2012} shows that this condition is necessary for a PDSC network to have a positive steady-state.
The basis vectors $\bfb^1$, \ldots, $\bfb^d$ from the proof of Theorem \ref{thrm:toriccycles} are of a special form. In particular, when the reaction rates $\kappa_i$ are positive, their nonzero entries are all positive. This shows that PDSC cycle networks automatically satisfy Condition \ref{cond:positiveSS}.

\begin{corollary} Let $G$ be a directed cycle. If $G$ is a PDSC network, then the chemical reaction system given by $G$ satifies Condition \ref{cond:positiveSS}.
\end{corollary}

\comment{
\begin{rmk}
    We note that the basis vectors $\bfb^1$, \ldots, $\bfb^d$ are of a special form. In particular, their nonzero entries are all positive. While it is not the focus of this paper, the authors of \cite{PDSC2012} present another condition (Condition 3.4,\cite{PDSC2012}) that guarantees the steady-states of the system (\ref{eq:massactionsystem}) are positive. The condition states that for each $i$, the nonzero entries of $\bfb^i$ have the same sign. The above proof shows that cycles satisfying Condition \ref{cond:toricsteadystates} automatically satisfy this other condition.
\end{rmk}
}

\begin{example}[Species-overlappling cycles] An instance of PDSC cycle networks are what we call \textit{species-overlapping cycles}, denoted $SOC_m$ where $m \geq 3$. This one-parameter family of cycle networks are defined by reactions of the form $$\begin{tikzcd}[column sep = small] X_{i} + X_{i+1} \ar[r,"\kappa_i"] & X_{i+1} + X_{i+2} \end{tikzcd}$$ for $i = 1, \ldots, m$ where the indices are taken modulo the set $[m] = \{1, \ldots, m\}$. For example, when $m=4$ we get the network seen in Figure $\ref{fig:toric 4-cycle}$. 

\begin{figure}[h]
    \centering
    \begin{tikzcd}
    A+B \arrow[r,"\kappa_1"] & B+C \arrow[d,"\kappa_2"]\\
    A+D \ar[u,"\kappa_4"] & C+D \ar[l,"\kappa_3"]
    \end{tikzcd}
    \caption{A four-cycle with positive binomial steady-states.}
    \label{fig:toric 4-cycle}
\end{figure}
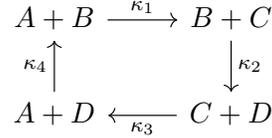

Note that the system of ordinary differential equations arising from these cycles is binomial. Our claim is that for $m\geq 3$ these networks are also indeed PDSC networks. When $m$ is odd, the matrix

\[
Y^t = 
\begin{pmatrix}
1 & & & & 1\\
1 & 1 & & & \\
 & 1 & 1 & & \\
& & & \ddots & \\
& & & & 1
\end{pmatrix} = \begin{pmatrix}
\bfe_1+\bfe_2 & \bfe_2+\bfe_3 & \cdots & \bfe_{m-1}+\bfe_m & \bfe_1 + \bfe_m
\end{pmatrix}
\]
has full rank and hence $x_{\boldsymbol\kappa} = ({\kappa}^{-1}_1, {\kappa}^{-1}_2, \ldots, {\kappa}^{-1}_m)^t$ generates the kernel of $\Sigma$. Thus, $d = 1$ and by Theorem \ref{thrm:toriccycles}, $SOC_m$ for odd $m$ is a PDSC network since $H(\ell) = T(\ell) = \emptyset$. Else if $m$ is even, then $d=2$ and a surjective coloring of the edges of the network is given as follows: $$\lambda(\bfy_i \to \bfy_{i+1}) = 
\begin{cases}
    1, &\text{ if $i$ is odd}\\
    2, &\text{ if $i$ is even}.
\end{cases}$$
With this coloring, we have $H(1) = T(2)$, $T(1) = H(2)$, and satisfy the following condition:
\begin{align*}
    \sum_{\bfy \in H(1)} \bfy
    &= \bfy_2 + \bfy_4 + \cdots + \bfy_m\\
    &= (\bfe_2+\bfe_3) + (\bfe_4+\bfe_5) + \cdots + (\bfe_m+\bfe_1)\\
    &= (\bfe_1 + \bfe_2) + (\bfe_3+\bfe_4) + \cdots + (\bfe_{m-1} + \bfe_m)\\
    &= \bfy_1 + \bfy_3 + \cdots + \bfy_{m-1}\\
    &= \sum_{\bfy \in T(1)} \bfy.
\end{align*}
Thus, by Theorem \ref{thrm:toriccycles} $SOC_m$ satisfies Condition \ref{cond:toricsteadystates} for even $m$. 
\end{example}

These networks are also partitionable and we compute the mixed volume of their natural system of equations as in Theorem \ref{thm:main}.

\begin{theorem}\label{thm:overlappingspecies}
Let $m \geq 3$. The cycle networks $SOC_m$ are partitionable. The mixed volumes of the associated systems
\[ 
\begin{cases}
    f_i &= \kappa_{i-2}x_{i-2}x_{i-1} - \kappa_ix_ix_{i+1} \quad \text{, for } i = 1, \ldots, m-1\\
    0 &= x_1 + x_2 + \cdots x_m + c
\end{cases}
\]
for odd $m$ and
\[
\begin{cases}
    f_i &= \kappa_{i-2}x_{i-2}x_{i-1} - \kappa_ix_ix_{i+1} \quad \text{, for } i = 1, \ldots, m-2\\
    0 &= x_1 + x_3 + \cdots x_{m-1} + c_1\\
    0 &= x_2 + x_4 + \cdots x_{m} + c_2
\end{cases}
\]
for even $m$ are $1$ and $\frac{m}{2}$, respectively.
\end{theorem}

\begin{proof} For the cycle network $SOC_m$, the polynomials of the system (\ref{eq:massactionsystem}) are $f_i = \kappa_{i-2}x_{i-2}x_{i-1} - \kappa_ix_ix_{i+1}$ so $\Newt(f_i) = \conv \mathcal A_i$ where $\mathcal A_i = \{\bfe_{i-2}+\bfe_{i-1},\bfe_i+\bfe_{i+1}\}$. We organize the proof based on the parity of $m$.

First suppose $m$ is odd. Then the network $SOC_m$ has one conservation law given by the conservation law vector $\bfw = \1$. Since each $f_i$ is homogenous then $I = \langle f_1, \ldots,f_{m-1} \rangle$ is multihomogenous with respect to the multigrading given be $\bfw$, hence $SOC_m$ is partitionable. By Theorem \ref{thm:main} the mixed volume of $SOC_m$ is the absolute value of the determinant of the matrix with columns $\bfe_1$ and $\bfe_{i-2}+\bfe_{i-1}-\bfe_i-\bfe_{i+1}$ for $i=1, \ldots,m-1$. Since one of the columns of this matrix is $\bfe_1$, we focus on the determinant of the submatrix after removing this column and the first row. For $m=3$, the submatrix is $\begin{pmatrix}
    0 & -1\\
    1 & 0
\end{pmatrix}$ which has determinant 1, as desired. For $m \geq 5$, the submatrix has an LU-factorization with

$$L = \left (
\begin{array}{c|c}
    I_{(m-3) \times (m-3)} & \mathbf{0}\\ \\ \hline \\
    \begin{matrix}
        -1 & \cdots & (-1)^j \lceil \nicefrac{j}{2} \rceil & \cdots & \lceil \nicefrac{(m-3)}{2} \rceil \\
        -1 & \cdots & -(j \mod 2) & \cdots & 0
    \end{matrix} & 
    \begin{matrix}
        1 & 0 \\
        \nicefrac{1}{\lceil \nicefrac{(m-3)}{2} \rceil} & 1 
    \end{matrix}
\end{array}
\right )
$$
and 
$$U = \left (
\begin{array}{c|c}
    U_1 & U_2 \\ \hline
    \mathbf{0} & \begin{matrix}
         \lceil \nicefrac{(m-3)}{2} \rceil & -1\\
         0 & \nicefrac{1}{\lceil \nicefrac{(m-3)}{2} \rceil}
     \end{matrix}
\end{array}
\right )
$$
where $I_{(m-3)\times (m-3)}$ is the $(m-3) \times (m-3)$ identity matrix and the $i$th row of $\begin{pmatrix}
    U_1 & U_2
\end{pmatrix}$ is $\bfe_{i+2} + \bfe_{i+3} - \bfe_i - \bfe_{i+1}$ except the last row is $\bfe_{m-1} - \bfe_{m-3} - \bfe_{m-2}$. Therefore, the mixed volume is $\det(L)\det(U) = (-1)^{m-3} = 1$.

Now suppose $m$ is even. The network $SOC_m$ has two conservation laws given by the vectors $\bfw_1 = \bfe_1 + \bfe_3 + \cdots + \bfe_{m-1}$ and $\bfw_2 = \bfe_2 + \bfe_4 + \cdots + \bfe_m$. Then for any $i,j$ and $\bfa,\bfb \in \mathcal A_i$, $\bfa \cdot \bfw_j = \bfb \cdot \bfw_j = 1$ so $I = \langle f_1, \ldots, f_{m-2} \rangle$ is multihomogeneous with respect to the conservation law vectors $\bfw_1, \bfw_2$ and hence $SOC_m$ is partitionable. By Theorem \ref{thm:main} the mixed volume of $SOC_m$ is the absolute value of the determinant of the matrix with columns $\bfe_1$, $\bfe_2$, and $\bfe_{i-2}+\bfe_{i-1}-\bfe_i-\bfe_{i+1}$ for $i=1, \ldots,m-2$. Since the first two columns are $\bfe_1, \bfe_2$, we focus on the determinant of the submatrix after removing the first two rows and columns. For $m=4$, the submatrix is $\begin{pmatrix}
    1 & -1\\
    1 & 1\\
\end{pmatrix}$ which has determinant $\frac{m}{2} = 2$, as desired. For $m \geq 6$, up to a permutation matrix, we have the following LU-factorization with
$$L = \left (
\begin{array}{c|c|c}
    1 & \mathbf{0} & 0 \\ \hline
    \mathbf{0} & I_{(m-4)\times (m-4)} & \mathbf{0}\\ \hline
    1 & \begin{matrix} -1 & \cdots & (-1)^j \lceil \nicefrac{j}{2} \rceil & \cdots & \lceil \nicefrac{(m-4)}{2} \rceil \end{matrix} & 1
\end{array} \right )$$
and 
$$U = \left(
\begin{array}{ccc}
     1 & \mathbf{0} & 1 \\
     \mathbf{0} & U_1 & U_2 \\
     0 & \mathbf{0} & \nicefrac{m}{2}
\end{array}
\right)
$$
where the $i$th row of $\begin{pmatrix}
    U_1 & U_2
\end{pmatrix}$ is $\bfe_{i+2} + \bfe_{i+3} - \bfe_i - \bfe_{i+1}$ except the last two rows are $\bfe_{m-2} - \bfe_{m-4} - \bfe_{m-3}$ and $- \bfe_{m-3}  - \bfe_{m-2}$. Thus, the mixed volume of $SOC_m$ for $m$ even is $(-1)^{(m-4)}\frac{m}{2} = \frac{m}{2}$.
\end{proof}

\section{Discussion}

In the present work, we gave a formula for the mixed volume of a binomial steady-state system for any chemical reaction network with partitionable conservation laws. This result was obtained by analyzing the possible structure of a fine mixed subdivision of the Minkowski sum of Newton polyopes from this system. An advantage of this approach is that it allows us to avoid computing a fine mixed subdivision, which is quite computationally expensive. We also characterized the directed cycles which are PDSC networks using edge colorings. We used these result to calculate the mixed volumes of all species-overlapping cycles. 

There are still many directions for further exploration. It would be interesting to consider the ways in which one can relax the disjoint support assumption for partitionable conservation laws. If one removes this assumption, there are many examples of fine mixed subdivisions with more than one cell of type $(1,\dots,1)$. Is it possible to characterize the number and volume of these in a fine mixed subdivision of such a network? Alternatively, one may search for a geometric algorithm for changing a non-partitionable network into a partitionable one and tracking the solutions. For instance, it can be shown that a slight modification of the Edelstein network is partitionable. Of particular interest, this modification can be realized by geometric means. To explain further, recall that a requirement for a partitionable network is that the affine hull of Newton polytopes $\Newt(f_i)$ are parallel to the stoichiometric subspace. In the case of the Edelstein network, it can be shown that there is an oblique projection of the affine hull of $\Newt(f_i)$ onto an affine space parallel to the stoichiometric subspace and this projection corresponds to a modification of the Edelstein network into a partitionable network while preserving the stoichiometric matrix. Thus, we are curious if this type of geometric argument can be made more general and how both the mixed volume and the steady-state degree compare to the respective quantities of the original system.

Theorem \ref{thrm:toriccycles} only applies to directed cycles and does not allow for bidirected edges. In the future, it would be interesting to generalize this result to networks whose underlying undirected graph is a cycle, and to determine to what extent we can use these results to ``glue" cycles together to create more complex networks.



\section{Acknowledgements}
Elizabeth Gross was supported by the National Science Foundation (NSF), DMS-1945584.

\bibliographystyle{plain}
\bibliography{references}
\end{document}